\documentclass[12pt,reqno]{article}

\usepackage[usenames]{color}
\usepackage{amssymb}
\usepackage{amsmath}
\usepackage{amsthm}
\usepackage{amsfonts}
\usepackage{amscd}
\usepackage{graphicx}

\usepackage[colorlinks=true,
linkcolor=webgreen,
filecolor=webbrown,
citecolor=webgreen]{hyperref}

\definecolor{webgreen}{rgb}{0,.5,0}
\definecolor{webbrown}{rgb}{.6,0,0}

\usepackage{color}
\usepackage{fullpage}
\usepackage{float}

\usepackage{graphics}
\usepackage{latexsym}
\usepackage{epsf}
\usepackage{accents}

\newcommand{\red}[1]{{\color{red}#1}}

\begin{document}

\theoremstyle{plain}
\newtheorem{theorem}{Theorem}
\newtheorem{corollary}[theorem]{Corollary}
\newtheorem{lemma}{Lemma}
\newtheorem{example}{Example}

\begin{center}
\vskip 1cm{\LARGE\bf 
On Ledin and Brousseau's summation problems\\
}
\vskip 1cm
\large
Kunle Adegoke \\
Department of Physics and Engineering Physics\\
Obafemi Awolowo University\\
220005 Ile-Ife, Nigeria\\
\href{mailto:adegoke00@gmail.com}{\tt adegoke00@gmail.com} \\
\end{center}

\vskip .2 in



\begin{abstract}
\noindent We develop a recursive scheme, as well as polynomial forms (polynomials in $n$ of degree $m$), for the evaluation of Ledin and Brousseau's Fibonacci sums of the form $S(m,n,r)=\sum_{k=1}^nk^mF_{k + r}$, $T(m,n,r)=\sum_{k=1}^nk^mL_{k + r}$ for non-negative integers $m$ and $n$ and arbitrary integer $r$; $F_j$ and $L_j$ being the $j^{th}$ Fibonacci and Lucas numbers. We also extend the study to a general second order sequence by establishing a recursive procedure to determine $W(m,n,r;a,b,p,q)=\sum_{k=1}^nk^mw_{k+r}$ where $(w_j(a,b;p,q))$ is the Horadam sequence defined by $w_0  = a,\,w_1  = b;\,w_j  = pw_{j - 1}  - qw_{j - 2}\, (j \ge 2);$ where $a$, $b$, $p$ and $q$ are arbitrary complex numbers, with $p\ne 0$ and $q\ne 0$. An explicit polynomial form for $W(m,n,r;a,b,1,q)$ and more generally for the sum $\mathcal W(m,n,h,r;a,b,p,q) = \sum_{k = 1}^n {V_h^{- k}k^m w_{hk + r}}$, where $(V_j(p,q))=(w_j(2,p;p,q))$, is established. Finally a polynomial form is established for a Ledin-Brousseau sum involving Horadam numbers with subscripts in arithmetic progression.
\end{abstract}

\section{Introduction}
Erbacher and Fuchs~\cite{sol64}, Ledin~\cite{ledin67}, Brousseau~\cite{brousseau67}, Zeitlin~\cite{zeitlin67} and recently Ollerton and Shannon~\cite{shannon21,ollerton20} and Dresden~\cite{dresden22} have developed various methods, including linear operator techniques, linear recurrence relations, finite differences approach and matrix methods, to study Fibonacci sums of the form
\[
S(m,n)=S(m,n,0)=\sum_{k=1}^nk^mF_k,\quad T(m,n)=T(m,n,0)=\sum_{k=1}^nk^mL_k,
\]
where $F_j$ and $L_j$ are the $j^{th}$ Fibonacci and Lucas numbers and $m$ and $n$ are non-negative integers. The sums $S(0,n)$, $T(0,n)$, $S(1,n)$, $T(1,n)$ are well-known. The sum $S(3,n)$, proposed as a problem by Brother U.~Alfred~\cite{prob63}, was later evaluated by Erbacher and Fuchs and also Dresner and Bicknell~\cite{sol64}. For values of $S(m,n)$ for $m=0,1,2,\ldots, 10$ and $T(m,n)$ for $m=0,1,2,\ldots, 5$, the interested reader may see Ledin~\cite[Table I, Table III]{ledin67}; noting that some corrections for Ledin's Table I are provided by Shannon and Ollerton~\cite[p.48]{shannon21}. 

Ledin~\cite{ledin67} has shown that $S(m,n)$ and $T(m,n)$ can be expressed in the form
\[
S(m,n) = P_1 (m,n)F_n  + P_2 (m,n)F_{n + 1}  + C(m),
\]

\[
T(m,n) = P_1 (m,n)L_n  + P_2 (m,n)L_{n + 1}  + K(m),
\]
where $P_1(m,n)$ and $P_2(m,n)$ are polynomials in $n$ of degree $m$ and $C(m)$ and $K(m)$ are constants depending only on $m$. Ledin gave some properties of $P_1$ and $P_2$ and developed a scheme for obtaining them by simple integration.

We will establish a recurrence relation for each of $P_1(m,n)$, $P_2(m,n)$, $C(m)$ and $K(m)$ through which $S(m,n)$ and $T(m,n)$ can then be determined. Specifically we will show that
\[
P_1 (m,n) = (n + 2)^m  - \sum_{j = 0}^{m - 1} {\binom mj(2^{m - j}  + 1)P_1 (j,n)},
\]

\[
P_2 (m,n) = (n + 1)^m  - \sum_{j = 0}^{m - 1} {\binom mj(2^{m - j}  + 1)P_2 (j,n)},
\]

\[
C (m) =- 1  - \sum_{j = 0}^{m - 1} {\binom mj(2^{m - j}  + 1)C (j)},
\]
and
\[
K (m) = -(2^{m + 1} + 1)  - \sum_{j = 0}^{m - 1} {\binom mj(2^{m - j}  + 1)K (j)}.
\]
Note that, in view of the definition of the empty sum~\eqref{eq.dcfw5vp}, the above recursive formulas already subsume the initial conditions: $P_1 (0,n)=1$, $P_2 (0,n)=1$, $C(0)=-1$, $K(0)=-3$. Our approach is different from that of Shannon and Ollerton~\cite{shannon21} who used algebraic methods to show that $S(m, n)$ satisfies a linear recurrence relation. Their expressions involve certain matrix elements and are rather complicated.

In the sequel, we will establish explicit polynomial forms for $P_1 (m,n)$ and $P_2 (m,n)$ and obtain the constants $C(m)$ and $K(m)$ by showing that
\[
P_1 (m,n) = n^m  + \sum\limits_{s = 1}^m {( - 1)^{s} \binom msn^{m - s} \sum\limits_{j = 1}^s {A(s,j)F_{j + s} } },
\]

\[
P_2 (m,n) = n^m  + \sum\limits_{s = 1}^m {( - 1)^{s} \binom msn^{m - s} \sum\limits_{j = 1}^s {A(s,j)F_{j + s + 1} } },
\]

\[
C(m) =  - \delta _{m,0} F_0  + ( - 1)^{m + 1} \sum_{j = 0}^m {A(m,j)F_{j + m + 1} } ,
\]

\[
K(m) =  - \delta _{m,0} L_0  + ( - 1)^{m + 1} \sum_{j = 0}^m {A(m,j)L_{j + m  + 1} }, 
\]
where $\delta_{ij}$ is Kronecker delta and $A(i,j)$ are Eulerian numbers (OEIS  \htmladdnormallink{A123125}{https://oeis.org/A123125}) defined, for non-negative integers $i$ and $j$, by
\begin{equation}\label{eq.pr9b061}
A(i,j)=\sum_{t=0}^j{(-1)^t\binom{i + 1}t(j - t)^i},
\end{equation}
with $A(0,0)=1$, $A(i,0)=0$ for $i\ge 1$.

We will also give polynomial forms for $S(m,n)$ and $T(m,n)$, for $m$ and $n$ non-negative integers, namely,
\[
\begin{split}
S(m,n) &= \sum_{k = 1}^n {k^m F_k }\\ 
 &= -\delta _{m,0}F_0 + n^m F_{n + 2}  + ( - 1)^{m + 1} \sum_{j = 0}^m {A(m,j)F_{j + m + 1} }\\
&\quad- \sum_{s = 1}^m {( - 1)^{s + 1} \binom msn^{m - s} \sum_{j = 1}^s {A(s,j)F_{j + n + s + 1} } }, 
\end{split}
\]

\[
\begin{split}
T(m,n) &= \sum_{k = 1}^n {k^m L_k }\\ 
 &=-\delta _{m,0}L_0 + n^m L_{n + 2}  + ( - 1)^{m + 1} \sum_{j = 0}^m {A(m,j)L_{j + m + 1} }\\
&\quad- \sum_{s = 1}^m {( - 1)^{s + 1} \binom msn^{m - s} \sum_{j = 1}^s {A(s,j)L_{j + n + s + 1} } }. 
\end{split}
\]
The Fibonacci numbers, $F_n$, and the Lucas numbers, $L_n$, are defined, for \text{$n\in\mathbb Z$}, through the recurrence relations 
\[
F_n=F_{n-1}+F_{n-2}, \text{($n\ge 2$)},\quad\text{$F_0=0$, $F_1=1$};
\]
and
\[
L_n=L_{n-1}+L_{n-2}, \text{($n\ge 2$)},\quad\text{$L_0=2$, $L_1=1$};
\]
with
\[
F_{-n}=(-1)^{n-1}F_n,\quad L_{-n}=(-1)^nL_n.
\]
Throughout this paper, we denote the golden ratio, $(1+\sqrt 5)/2$, by $\alpha$ and write $\beta=(1-\sqrt 5)/2=-1/\alpha$, so that $\alpha\beta=-1$ and $\alpha+\beta=1$. 

Explicit formulas (Binet formulas) for the Fibonacci and Lucas numbers are
\[
F_n  = \frac{{\alpha ^n  - \beta ^n }}{{\alpha  - \beta }},\quad L_n  = \alpha ^n  + \beta ^n,\quad n\in\mathbb Z.
\]
Koshy \cite{koshy} and Vajda \cite{vajda} have written excellent books dealing with Fibonacci and Lucas numbers.

We will extend the study of the Ledin and Brousseau's summation problems to the Horadam sequence~\cite{horadam65}, $(w_j) = \left(w_j(a,b; p, q)\right)$, defined by the recurrence relation
\begin{equation}\label{eq.vhrb5b3}
w_0  = a,\,w_1  = b;\,w_j  = pw_{j - 1}  - qw_{j - 2}\, (j \ge 2);
\end{equation}
where $a$, $b$, $p$ and $q$ are arbitrary complex numbers, with $p\ne 0$, $q\ne 0$ and $p\ne q + 1$.

Two important cases of  $(w_n)$ are the Lucas sequences of the first kind, $(U_n(p,q))=(w_n(0,1;p,q))$, and of the second kind, $(V_n(p,q))=(w_n(2,p;p,q))$; so that 
\[
\mbox{$U_0=0$, $U_1=1$};\, U_n=pU_{n-1}-qU_{n-2}, \mbox{($n\ge 2$)};
\]
and
\[
\mbox{$V_0=2$, $V_1=p$};\, V_n=pV_{n-1}-qV_{n-2}, \mbox{($n\ge 2$)}.
\]
The most well-known Lucas sequences are the Fibonacci sequence, $(F_n)=(U_n(1,-1))$ and the sequence of Lucas numbers, $(L_n)=(V_n(1,-1)$.

Extension of the definition of $w_n$ to negative subscripts is provided by writing the recurrence relation as $w_{-n}=(pw_{-n+1}-w_{-n+2})/q$.

We will establish a recursive procedure to evaluate $W(m,n,r;a,b,p,q)=\sum_{k=1}^nk^mw_{k+r}$. An explicit polynomial form will be developed for $W(m,n,0;a,b,1,q)$ and more generally for the sum $\mathcal W(m,n,h,r;a,b,p,q) = \sum_{k = 1}^n {V_h^{- k}k^m w_{hk + r}}$. Finally, we will evaluate the Ledin-Brousseau summation for the Horadam sequence with indices in arithmetic progression by showing that
\[
\begin{split}
&W(m,n,r,h;a,b,p,q)=\sum_{k = 1}^n {k^m w_{hk + r} }\\
&  =  - \delta _{m,0} w_r  - n^m \left( {\frac{{w_{h(n + 1) + r}  - q^h w_{hn + r} }}{{1 - V_h  + q^h }}} \right)\\
&\qquad + \frac1{(1 - V_h  + q^h )^{m + 1} }\sum_{c = 0}^{m + 1} {( - 1)^c \binom{m + 1}cq^{hc} \sum_{j = 0}^m {A(m,j)w_{h(j - c) + r} } } \\
&\qquad\quad- \sum_{s = 1}^m {\binom ms\frac{{n^{m - s} }}{{(1 - V_h  + q^h )^{s + 1} }}\sum_{c = 0}^{s + 1} {( - 1)^c \binom{s + 1}cq^{hc} \sum_{j = 0}^s {A(s,j)w_{h(j - c + n) + r} } } } .
\end{split}
\]
Throughout this paper we assume $0^0=1$ and take the empty sum as
\begin{equation}\label{eq.dcfw5vp}
\sum_{k=j}^{j - 1}{f_k}=0,
\end{equation}
for any arbitrary sequence $(f_i)$.
\section{A recursive relation for each of $P_1(m,n)$, $P_2(m,n)$, $C(m)$ and $K(m)$}
With $z=e^yx$, $x,y\in\mathbb R$, in the geometric progression summation identity
\[
\sum_{k=0}^nz^k = \frac{z^{n+1}-1}{z-1},\quad z\ne 1,
\]
we can define
\[
R(x,y,n) = \sum_{k = 0}^n {e^{ky} x^k }  = \frac{{e^{y( n + 1)} x^{n + 1}  - 1}}{{e^y x - 1}}.
\]
Let
\begin{equation}\label{eq.bcycobv}
Q(x,m,n) = \sum_{k = 0}^n {k^m x^k };
\end{equation}
so that
\begin{equation}\label{eq.idod6y1}
Q(x,m,n) = \left. {\frac{{\partial ^m }}{{\partial y^m }}R(x,y,n)} \right|_{y = 0}  = \left. {\frac{{\partial ^m }}{{\partial y^m }}\frac{{e^{y( n + 1)} x^{n + 1}  - 1}}{{e^y x - 1}}} \right|_{y = 0}.
\end{equation}
\begin{lemma}
For non-negative integers $m$ and $n$,
\[\tag{F1}
S(m,n) = \sum_{k = 1}^n {k^m F_k }  = \left. {\frac{{\partial ^m }}{{\partial y^m }}\left( {\frac{{e^{y(n + 2)} F_n  + e^{y(n + 1)} F_{n + 1}  - e^y }}{{e^{2y}  + e^y  - 1}}} \right)} \right|_{y = 0},
\]

\[\tag{L1}
T(m,n) = \sum_{k = 1}^n {k^m L_k }  = -L_0\,\delta _{m,0} + \left. {\frac{{\partial ^m }}{{\partial y^m }}\left( {\frac{{e^{y(n + 2)} L_n  + e^{y(n + 1)} L_{n + 1}  + e^y  - 2}}{{e^{2y}  + e^y  - 1}}} \right)} \right|_{y = 0}.
\]
\end{lemma}
\begin{proof}
From~\eqref{eq.bcycobv} we have
\[\tag{H1}
Q(\alpha ,m,n) = \delta _{m,0}  + \sum_{k = 1}^n {k^m \alpha ^k } ,
\]
\[\tag{H2}
Q(\beta ,m,n) = \delta _{m,0}  + \sum_{k = 1}^n {k^m \beta ^k } ,
\]
from which we get
\begin{equation}\label{eq.ecnfdhw}
Q(\alpha ,m,n) - Q(\beta ,m,n) = \sum_{k = 1}^n {k^m (\alpha ^k  - \beta ^k )} +0\,\delta _{m,0}  = S(m,n)\sqrt 5 + F_0\delta _{m,0}\sqrt 5
\end{equation}
and
\[
Q(\alpha ,m,n) + Q(\beta ,m,n) = \sum_{k = 1}^n {k^m (\alpha ^k  + \beta ^k )} + 2\,\delta _{m,0} = T(m,n) + L_0\delta _{m,0}.
\]
But, using~\eqref{eq.idod6y1} and the linearity of partial differentiation, we have
\[
\begin{split}
Q(\alpha ,m,n) - Q(\beta ,m,n) &= \left. {\frac{{\partial ^m }}{{\partial y^m }}\left( {\frac{{(e^y \alpha )^{n + 1}  - 1}}{{e^y \alpha  - 1}} - \frac{{(e^y \beta )^{n + 1}  - 1}}{{e^y \beta  - 1}}} \right)} \right|_{y = 0}\\
&= \left. {\frac{{\partial ^m }}{{\partial y^m }}\left( {\frac{{(e^y \beta  - 1)((e^y \alpha )^{n + 1}  - 1) - (e^y \alpha  - 1)((e^y \beta )^{n + 1}  - 1)}}{{(e^y \alpha  - 1)(e^y \beta  - 1)}}} \right)} \right|_{y = 0},
\end{split}
\]
and
\[
\begin{split}
Q(\alpha ,m,n) + Q(\beta ,m,n) &= \left. {\frac{{\partial ^m }}{{\partial y^m }}\left( {\frac{{(e^y \alpha )^{n + 1}  - 1}}{{e^y \alpha  - 1}} + \frac{{(e^y \beta )^{n + 1}  - 1}}{{e^y \beta  - 1}}} \right)} \right|_{y = 0}\\
&= \left. {\frac{{\partial ^m }}{{\partial y^m }}\left( {\frac{{(e^y \beta  - 1)((e^y \alpha )^{n + 1}  - 1) + (e^y \alpha  - 1)((e^y \beta )^{n + 1}  - 1)}}{{(e^y \alpha  - 1)(e^y \beta  - 1)}}} \right)} \right|_{y = 0};
\end{split}
\]
from which upon clearing brackets, re-arranging the terms and using the Binet formulas we get
\begin{equation}\label{eq.q1yc93r}
Q(\alpha ,m,n) - Q(\beta ,m,n) = \left. {\frac{{\partial ^m }}{{\partial y^m }}\left( {\frac{{ - e^{y(n + 2)} F_n  + e^y  - e^{y(n + 1)} F_{n + 1} }}{{ - e^{2y}  - e^y  + 1}}} \right)} \right|_{y = 0} \sqrt 5
\end{equation}
and
\[
Q(\alpha ,m,n) + Q(\beta ,m,n) = \left. {\frac{{\partial ^m }}{{\partial y^m }}\left( {\frac{{ -e^{y(n + 2)} L_n  - e^{y(n + 1)} L_{n + 1}  - e^y  + 2}}{{ - e^{2y}  - e^y  + 1}}} \right)} \right|_{y = 0}.
\]
Comparing~\eqref{eq.ecnfdhw} and~\eqref{eq.q1yc93r} we get (F1). The proof of (L1) is similar. 
\end{proof}
Clearly, (F1) and (L1) can be written in the Ledin form
\[
S(m,n) = P_1 (m,n)F_n  + P_2 (m,n)F_{n + 1}  + C(m),
\]

\[
T(m,n) = P_1 (m,n)L_n  + P_2 (m,n)L_{n + 1}  + K(m),
\]
with
\begin{equation}\label{eq.u7ul4a4}
P_1 (m,n) = \left. {\frac{{\partial ^m }}{{\partial y^m }}\frac{{e^{y(n + 2)} }}{{e^{2y}  + e^y  - 1}}} \right|_{y = 0} ,\quad P_2 (m,n) = \left. {\frac{{\partial ^m }}{{\partial y^m }}\frac{{e^{y(n + 1)} }}{{e^{2y}  + e^y  - 1}}} \right|_{y = 0},
\end{equation}
and
\begin{equation}\label{eq.ry2tgvo}
C(m) =  - \left. {\frac{{\partial ^m }}{{\partial y^m }}\frac{{e^y }}{{e^{2y}  + e^y  - 1}}} \right|_{y = 0} ,\quad K(m) = \left. {\frac{{\partial ^m }}{{\partial y^m }}\frac{{e^y  - 2}}{{e^{2y}  + e^y  - 1}}} \right|_{y = 0} - L_0\,\delta _{m,0}.
\end{equation}
Note that the first identity in \eqref{eq.ry2tgvo} already answers the question raised in the concluding comments of Ollerton and Shannon~\cite{ollerton20} concerning the relationship between the Ledin form and \mbox{$e^{-y}/(1-e^{y}+e^{-y})$}.

We are now in a position to state our first main result.
\begin{theorem}\label{thm.o2yjlhn}
Let $m$ and $n$ be non-negative integers. Then,
\[\tag{F}
S(m,n) = \sum_{k=1}^nk^mF_k = P_1 (m,n)F_n  + P_2 (m,n)F_{n + 1}  + C(m),
\]

\[\tag{L}
T(m,n) = \sum_{k=1}^nk^mL_k = P_1 (m,n)L_n  + P_2 (m,n)L_{n + 1}  + K(m),
\]
where $P_1(m,n)$, $P_2(m,n)$, $C(m)$ and $K(m)$ are given recursively by
\[
P_1 (m,n) = (n + 2)^m  - \sum_{j = 0}^{m - 1} {\binom mj(2^{m - j}  + 1)P_1 (j,n)},
\]

\[
P_2 (m,n) = (n + 1)^m  - \sum_{j = 0}^{m - 1} {\binom mj(2^{m - j}  + 1)P_2 (j,n)},\]

\[
C (m) =- 1  - \sum_{j = 0}^{m - 1} {\binom mj(2^{m - j}  + 1)C (j)},
\]
and
\[
K (m) = -(2^{m + 1} + 1)  - \sum_{j = 0}^{m - 1} {\binom mj(2^{m - j}  + 1)K (j)}.
\]
\end{theorem}
\begin{proof}
We seek to perform the differentiations prescribed in \eqref{eq.u7ul4a4} and \eqref{eq.ry2tgvo}.

Let
\begin{equation}\label{eq.kq8eqc2}
U(y,n) = \frac{{e^{y(n + 2)} }}{{e^{2y}  + e^y  - 1}};
\end{equation}
so that, for non-negative integer $j$,
\begin{equation}\label{eq.t65mbvk}
P_1 (j,n) = \left. {\frac{{\partial ^j }}{{\partial y^j }}U(y,n)} \right|_{y = 0}.
\end{equation}
For brevity let $U\equiv U(y,n)$ and write \eqref{eq.kq8eqc2} as
\[
(e^{2y}  + e^y  - 1)U = e^{y(n + 2)}.
\]
Leibnitz rule for differentiation gives
\[
\sum_{j = 0}^m {\binom mj\frac{{\partial ^{m - j} }}{{\partial y^{m - j} }}(e^{2y}  + e^y  - 1)}\frac{{\partial ^j U}}{{\partial y^j }}  = \frac{{\partial ^m }}{{\partial y^m }}e^{y(n + 2)}.
\]
Thus,
\[
\begin{split}
&\sum_{j = 0}^{m - 1} {\binom mj(2^{m - j} e^{2y}  + e^y )\frac{{\partial ^j U}}{{\partial y^j }}}  + (e^{2y}  + e^y  - 1)\frac{{\partial ^m U}}{{\partial y^m }}\\
&\qquad= (n + 2)^m e^{y(n + 2)}.
\end{split}
\]
Evaluating both sides at $y=0$, making use of~\eqref{eq.t65mbvk}, we have
\[
\sum_{j = 0}^{m - 1} {\binom mjP_1 (j,n)(2^{m - j}  + 1)}  + P_1 (m,n) = (n + 2)^m;
\]
from which the $P_1(m,n)$ recurrence follows. A similar procedure gives $P_2(m,n)$. Finally, the $C(m)$ and $K(m)$ recurrence relations follow from
\[
C(m)= - P_2(m,0),\quad K(m)=- 2P_1(m,0) - P_2(m,0). 
\]
\end{proof}
The original sum in which Brousseau was interested is $\sum_{k=1}^n k^mF_{k + r}$. 

From (H1) and (H2) we find
\begin{gather}
\alpha ^r Q(\alpha ,m,n) - \beta ^r Q(\beta ,m,n) = \sqrt 5 \sum_{k = 1}^n {k^m F_{k + r} } +\sqrt 5\delta _{m,0}F_r,\label{eq.faka3lc}\\
\alpha ^r Q(\alpha ,m,n) + \beta ^r Q(\beta ,m,n) = \sum_{k = 1}^n {k^m L_{k + r} } + \delta _{m,0}L_r.\label{eq.qucw3r8}
\end{gather}
But from \eqref{eq.idod6y1} and the Binet formulas, we get
\begin{equation}\label{eq.xxeie37}
\begin{split}
&\alpha ^r Q(\alpha ,m,n) - \beta ^r Q(\beta ,m,n)\\
&= \left. {\frac{{\partial ^m }}{{\partial y^m }}\left( {\frac{{e^{y(n + 2)} F_{n + r}  + e^{y(n + 1)} F_{n + r + 1}  - e^y F_{r - 1}  - F_r }}{{e^{2y}  + e^y  - 1}}} \right)} \right|_{y = 0} \sqrt 5, 
\end{split}
\end{equation}

\begin{equation}\label{eq.tko9zi9}
\begin{split}
&\alpha ^r Q(\alpha ,m,n) + \beta ^r Q(\beta ,m,n)\\
&= \left. {\frac{{\partial ^m }}{{\partial y^m }}\left( {\frac{{e^{y(n + 2)} L_{n + r}  + e^{y(n + 1)} L_{n + r + 1}  - e^y L_{r - 1}  - L_r }}{{e^{2y}  + e^y  - 1}}} \right)} \right|_{y = 0}. 
\end{split}
\end{equation}
Comparing \eqref{eq.faka3lc} and \eqref{eq.xxeie37} and \eqref{eq.qucw3r8} and \eqref{eq.tko9zi9}, we find
\begin{equation}\label{eq.xpcud2p}
S(m,n,r) = \sum_{k=1}^nk^mF_{k + r} = P_1 (m,n)F_{n + r}  + P_2 (m,n)F_{n + r + 1}  + C(m,r),
\end{equation}

\begin{equation}\label{eq.ktjpxxp}
T(m,n,r) = \sum_{k=1}^nk^mL_{k + r}= P_1 (m,n)L_{n + r}  + P_2 (m,n)L_{n + r + 1}  + K(m,r),
\end{equation}
where $P_1(m,n)$ and $P_2(m,n)$ are the same as in Theorem~\ref{thm.o2yjlhn} and $C(m,r)$ and $K(m,r)$ are given by
\[
C(m,r) =  - \left. {\frac{{\partial ^m }}{{\partial y^m }}\left( {\frac{{e^y F_{r - 1}  + F_r }}{{e^{2y}  + e^y  - 1}}} \right)} \right|_{y = 0}  - \delta _{m,0} F_r ,
\]

\[
K(m,r) =  - \left. {\frac{{\partial ^m }}{{\partial y^m }}\left( {\frac{{e^y L_{r - 1}  + L_r }}{{e^{2y}  + e^y  - 1}}} \right)} \right|_{y = 0}  - \delta _{m,0} L_r ,
\]
and can be found directly from \eqref{eq.xpcud2p}, \eqref{eq.ktjpxxp} and the recurrence relations for $P_1$ and $P_2$. Thus,
\[
\begin{split}
C(m,r)& =  - P_1 (m,0)F_r  - P_2 (m,0)F_{r + 1}\\ 
&=  - 2^m F_r  - F_{r + 1}  - \sum\limits_{j = 0}^{m - 1} {\binom mj(2^{m - j}  + 1)C(j,r)}, 
\end{split}
\]

\[
\begin{split}
K(m,r)& =  - P_1 (m,0)L_r  - P_2 (m,0)L_{r + 1}\\ 
&=  - 2^m L_r  - L_{r + 1}  - \sum\limits_{j = 0}^{m - 1} {\binom mj(2^{m - j}  + 1)K(j,r)}. 
\end{split}
\]
\section{Polynomial forms for $S(k,m)$ and $T(k,m)$}
Let $x$ be a real or complex number such that $x\ne0$ and $x\ne 1$; and let $m$ and $n$ be non-negative integers. Hsu and Tan~\cite{hsu00} have shown that
\begin{equation}\label{eq.k7m77ev}
\begin{split}
Q(x,m,n) &= \sum_{k = 0}^n {k^m x^k }\\
 &=  - n^m \frac{{x^{n + 1} }}{{1 - x}} + \frac{{A_m (x)}}{{(1 - x)^{m + 1} }} - \sum_{s = 1}^m {x^n \binom msn^{m - s} \frac{{A_s (x)}}{{(1 - x)^{s + 1} }}}, 
\end{split}
\end{equation}
where
\begin{equation}\label{eq.qjlois5}
A_i (x) = \sum_{j = 0}^i {A(i,j)x^j }, \quad i\in\mathbb N_0, \quad A_0(x)=1,
\end{equation}
where $A(i,j)$ are the Eulerian numbers defined in \eqref{eq.pr9b061}.
\begin{lemma}\label{lem.bvh0lkh}
If $j$ is a non-negative integer and $s$ is any integer, then,
\begin{gather}
\alpha ^s A_j (\alpha ) - \beta ^s A_j (\beta ) = \sqrt 5 \sum_{t = 0}^j {A(j,t)F_{t + s} },\label{eq.tytwb4a}\\
\alpha ^s A_j (\alpha ) + \beta ^s A_j (\beta ) = \sum_{t = 0}^j {A(j,t)L_{t + s} }\label{eq.i71ftvx}. 
\end{gather}
\end{lemma}
\begin{proof}
Let
\[
f=\alpha ^s A_j (\alpha ) - \beta ^s A_j (\beta ).
\]
Since $\alpha^s=L_s - \beta^s$, we have
\begin{equation}\label{eq.s7llesz}
f = L_s A_j (\alpha ) - \beta ^s (A_j (\alpha ) + A_j (\beta )).
\end{equation}
Similarly, we find
\begin{equation}\label{eq.m08ht0e}
f =  - L_s A_j (\beta ) + \alpha ^s (A_j (\alpha ) + A_j (\beta )).
\end{equation}
Addition of \eqref{eq.s7llesz} and \eqref{eq.m08ht0e} gives
\[
2f = F_s \sqrt 5 (A_j (\alpha ) + A_j (\beta )) + L_s (A_j (\alpha ) - A_j (\beta )).
\]
Thus,
\[
\begin{split}
&\alpha ^s A_j (\alpha ) - \beta ^s A_j (\beta )\\
 &= \frac{{F_s \sqrt 5 }}{2}(A_j (\alpha ) + A_j (\beta )) + \frac{{L_s }}{2}(A_j (\alpha ) - A_j (\beta ))\\
 &= \frac{{F_s \sqrt 5 }}{2}\sum_{t = 0}^j {A(j,t)L_t }  + \frac{{L_s \sqrt 5 }}{2}\sum_{t = 0}^j {A(j,t)F_t },\quad\mbox{by \eqref{eq.qjlois5}},\\ 
 &= \sqrt 5 \sum_{t = 0}^j {A(j,t)\frac{{F_s L_t  + F_t L_s }}{2}},
\end{split}
\]
from which identity \eqref{eq.tytwb4a} follows. The proof of \eqref{eq.i71ftvx} is similar.
\end{proof}
\begin{theorem}\label{thm.gmjtnyz}
If $m$ and $n$ are non-negative integers, then,
\begin{equation}\label{eq.t8yp45b}
\begin{split}
S(m,n) &= \sum_{k = 1}^n {k^m F_k }=\sum_{k = 0}^n {k^m F_k } - \delta_{m,0}F_0\\ 
 &=- \delta_{m,0}F_0 + n^m F_{n + 2}  + ( - 1)^{m + 1} \sum_{j = 0}^m {A(m,j)F_{j + m + 1} }\\
&\quad- \sum_{s = 1}^m {( - 1)^{s + 1} \binom msn^{m - s} \sum_{j = 1}^s {A(s,j)F_{j + n + s + 1} } }, 
\end{split}
\end{equation}

\begin{equation}\label{eq.skp52tl}
\begin{split}
T(m,n) &= \sum_{k = 1}^n {k^m L_k } = \sum_{k = 0}^n {k^m L_k } - \delta_{m,0}L_0\\ 
 &= - \delta_{m,0}L_0 + n^m L_{n + 2}  + ( - 1)^{m + 1} \sum_{j = 0}^m {A(m,j)L_{j + m + 1} }\\
&\quad- \sum_{s = 1}^m {( - 1)^{s + 1} \binom msn^{m - s} \sum_{j = 1}^s {A(s,j)L_{j + n + s + 1} } }, 
\end{split}
\end{equation}
where $A(i,j)$ are Eulerian numbers defined in \eqref{eq.pr9b061}.
\end{theorem}
\begin{proof}
From \eqref{eq.ecnfdhw} and \eqref{eq.k7m77ev} we have
\[
\begin{split}
&S(m,n)\sqrt 5 + F_0\delta _{m,0}\sqrt 5=Q(\alpha ,m,n) - Q(\beta ,m,n)\\
&= n^m \left( { - \frac{{\alpha ^{n + 1} }}{\beta } + \frac{{\beta ^{n + 1} }}{\alpha }} \right) + \left( {\frac{{A_m (\alpha )}}{{\beta ^{m + 1} }} - \frac{{A_m (\beta )}}{{\alpha ^{m + 1} }}} \right)\\
&\qquad- \sum_{s = 1}^m {n^{m - s} \binom ms\left( {\alpha ^n \frac{{A_s (\alpha )}}{{\beta ^{s + 1} }} - \beta ^n \frac{{A_s (\beta )}}{{\alpha ^{s + 1} }}} \right)}\\
& =- n^m \frac{(\alpha ^{n + 2}  - \beta ^{n + 2} )}{\alpha\beta} + \frac{{\alpha ^{m + 1} A_m (\alpha ) - \beta ^{m + 1} A_m (\beta )}}{{(\alpha \beta )^{m + 1} }}\\
&\qquad- \sum_{s = 1}^m {n^{m - s} \binom ms\left( {\frac{{\alpha ^{n + s + 1} A_s (\alpha ) - \beta ^{n + s + 1} A_s (\beta )}}{{(\alpha \beta )^{s + 1} }}} \right)}, 
\end{split}
\]
from which, using the Binet formula and identity \eqref{eq.tytwb4a}, we obtain identity \eqref{eq.t8yp45b}. The proof of identity \eqref{eq.skp52tl} proceeds in a similar way; we use
\[
T(m,n) + L_0\delta _{m,0}=Q(\alpha ,m,n) + Q(\beta ,m,n).
\]
\end{proof}
Using~\eqref{eq.faka3lc}, \eqref{eq.qucw3r8}, \eqref{eq.k7m77ev} and \eqref{eq.qjlois5}, the results in Theorem~\ref{thm.gmjtnyz} readily extend to the Brousseau sums $S(m,n,r)$ and $T(m,n,r)$ for non-negative integers $m$ and $n$ and any integer $r$. We have
\begin{equation}\label{eq.yn6pcdo}
\begin{split}
S(m,n,r)=\sum_{k = 1}^n {k^m F_{k + r} }&= -\delta _{m,0}F_r + n^m F_{n + r + 2}  + ( - 1)^{m + 1} \sum_{j = 0}^m {A(m,j)F_{j + m + r + 1} }\\
&\quad- \sum_{s = 1}^m {( - 1)^{s + 1} \binom msn^{m - s} \sum_{j = 1}^s {A(s,j)F_{j + n + s + r + 1} } }, 
\end{split}
\end{equation}

\begin{equation}\label{eq.hzvddni}
\begin{split}
T(m,n,r)=\sum_{k = 1}^n {k^m L_{k + r} }&= -\delta _{m,0}L_r + n^m L_{n + r + 2}  + ( - 1)^{m + 1} \sum_{j = 0}^m {A(m,j)L_{j + m + r + 1} }\\
&\quad- \sum_{s = 1}^m {( - 1)^{s + 1} \binom msn^{m - s} \sum_{j = 1}^s {A(s,j)L_{j + n + s + r + 1} } }. 
\end{split}
\end{equation}
Comparing identities \eqref{eq.xpcud2p} and \eqref{eq.yn6pcdo} and \eqref{eq.ktjpxxp} and \eqref{eq.hzvddni}, using equality of coefficients of equivalent polynomials in $n$, we deduce that the Ledin summation constants $C(m,r)$ and $K(m,r)$, $m\in\mathbb N_0$, $r\in\mathbb Z$, are given by
\begin{equation}\label{eq.t5z8cgr}
C(m,r) =  - \delta _{m,0} F_r  + ( - 1)^{m + 1} \sum_{j = 0}^m {A(m,j)F_{j + m + r + 1} } ,
\end{equation}
\begin{equation}\label{eq.lzeq4vs}
K(m,r) =  - \delta _{m,0} L_r  + ( - 1)^{m + 1} \sum_{j = 0}^m {A(m,j)L_{j + m + r + 1} }. 
\end{equation}
To conclude this section, we now determine the polynomial forms for $P_1(m,n)$ and $P_2(m,n)$.

Setting $r=-n - 1$ and $r=-n$, in turn, in \eqref{eq.xpcud2p}, we find
\begin{equation}\label{eq.a59m297}
P_1 (m,n) = S(m,n, - n - 1) - C(m, - n - 1),
\end{equation}

\begin{equation}\label{eq.dqcoy0r}
P_2 (m,n) = S(m,n, - n) - C(m, - n).
\end{equation}
Thus, from \eqref{eq.yn6pcdo}, \eqref{eq.t5z8cgr}, \eqref{eq.a59m297} and \eqref{eq.dqcoy0r} we obtain
\begin{equation}
P_1 (m,n) = n^m  + \sum\limits_{s = 1}^m {( - 1)^s \binom msn^{m - s} \sum\limits_{j = 1}^s {A(s,j)F_{j + s} } },
\end{equation}

\[
P_2 (m,n) = n^m  + \sum\limits_{s = 1}^m {( - 1)^s \binom msn^{m - s} \sum\limits_{j = 1}^s {A(s,j)F_{j + s + 1} } }.
\]
Thus, polynomial forms of $S(m,n,r)$ and $T(m,n,r)$ for non-negative integers $m$, $n$ and any integer $r$ are
\begin{equation}\label{eq.x7q4h09}
\begin{split}
&S(m,n,r)=\sum_{k=1}^n{k^mF_{k + r}}\\
&=- \delta _{m,0} F_r + \left(n^m  + \sum\limits_{s = 1}^m {( - 1)^{s} \binom msn^{m - s} \sum\limits_{j = 1}^s {A(s,j)F_{j + s} } }\right)F_{n + r}\\
&\qquad+ \left(n^m  + \sum\limits_{s = 1}^m {( - 1)^{s} \binom msn^{m - s} \sum\limits_{j = 1}^s {A(s,j)F_{j + s + 1} } }\right)F_{n + r + 1}\\
&\quad\qquad+ ( - 1)^{m + 1} \sum_{j = 0}^m {A(m,j)F_{j + m + r + 1} },
\end{split}
\end{equation}

\begin{equation}\label{eq.cvt8dsd}
\begin{split}
&T(m,n,r)=\sum_{k=1}^n{k^mL_{k + r}}\\
&=- \delta _{m,0} L_r + \left(n^m  + \sum\limits_{s = 1}^m {( - 1)^{s} \binom msn^{m - s} \sum\limits_{j = 1}^s {A(s,j)F_{j + s} } }\right)L_{n + r}\\
&\qquad+ \left(n^m  + \sum\limits_{s = 1}^m {( - 1)^{s} \binom msn^{m - s} \sum\limits_{j = 1}^s {A(s,j)F_{j + s + 1} } }\right)L_{n + r + 1}\\
&\quad\qquad+ ( - 1)^{m + 1} \sum_{j = 0}^m {A(m,j)L_{j + m + r + 1} }.
\end{split}
\end{equation}
\section{Extension to the Horadam sequence}
We now extend the study of the Ledin and Brousseau summation to the Horadam sequence. First we give the Horadam sequence version of Theorem~\ref{thm.o2yjlhn}, the Ledin form.
\begin{theorem}\label{thm.fnhvpm9}
Let $m$ and $n$ be non-negative integers. Then,
\[
W(m,n;a,b,p,q)=\sum_{k = 1}^n {k^m w_k }  = \mathcal P_1 (m,n;p,q)w_n  + \mathcal P_2 (m,n;p,q)w_{n + 1}  + \mathcal C(m;a,b,p,q);
\]
where
\[
(q - p + 1)\mathcal P_1 (m,n;p,q) = (n + 2)^m q - \sum_{j = 0}^{m - 1} {\binom mj(2^{m - j} q - p)\mathcal P_1 (j,n;p,q)},
\]

\[
(q - p + 1)\mathcal P_2 (m,n;p,q) = -(n + 1)^m - \sum_{j = 0}^{m - 1} {\binom mj(2^{m - j} q - p)\mathcal P_2 (j,n;p,q)},
\]

\[
(q - p + 1)\mathcal C (m;a,b,p,q) = -2^maq + b - \sum_{j = 0}^{m - 1} {\binom mj(2^{m - j} q - p)\mathcal C (j;a,b,p,q)}.
\]
\end{theorem}
\begin{proof}
It is known that
\[
w_n(a,b;p,q)=\mathbb A\red{(}\tau(p,q),\sigma(p,q),a,b\red{)}\tau(p,q)^n + \mathbb B\red{(}\tau(p,q),\sigma(p,q),a,b\red{)}\sigma(p,q)^n,
\]
or briefly~\cite{horadam65},
\begin{equation}\label{eq.nfmkcaz}
w_n = \mathbb A \tau^n + \mathbb B \sigma^n,
\end{equation}
where $\tau\equiv\tau(p,q)$ and $\sigma\equiv\sigma(p,q)$, $\tau\ne \sigma$, are the roots of the characteristic equation of the Horadam sequence, $x^2=px - q$; so that
\begin{equation}\label{eq.ufpxoag}
\tau=\frac{p+\sqrt{p^2-4q}}2,\quad \sigma=\frac{p-\sqrt{p^2-4q}}2,
\end{equation}
\[
\tau+\sigma=p,\quad\tau-\sigma=\sqrt{p^2-4q}\quad\mbox{and }\tau\sigma=q;
\]
and where $\mathbb A \equiv \mathbb A\red{(}\tau(p,q),\sigma(p,q),a,b\red{)}$ and $\mathbb B\equiv\mathbb B\red{(}\tau(p,q),\sigma(p,q),a,b\red{)}$ are given by
\[
\mathbb A=\frac{{b - a\sigma}}{{\tau  - \sigma }},\quad \mathbb B=\frac{{a\tau  - b}}{{\tau  - \sigma }}.
\]
Note that in the Fibonacci and Lucas cases, $\alpha=\tau(1,-1)$ and $\beta=\sigma(1,-1)$.

From~\eqref{eq.bcycobv} and \eqref{eq.idod6y1}, using~\eqref{eq.nfmkcaz} we find
\[
\begin{split}
&\mathbb AQ(\tau,m,n) + \mathbb BQ(\sigma,m,n)\\
&=\sum_{k = 0}^n {k^m w_k }  =\sum_{k = 1}^n {k^m w_k } +a\,\delta _{m,0}= \left. {\frac{{\partial ^m }}{{\partial y^m }}\frac{{(e^{y(n + 2)} qw_n  - e^{y(n + 1)} w_{n + 1}  + a - e^y (ap - b))}}{{e^{2y} q - e^y p + 1}}} \right|_{y = 0}.
\end{split}
\]
Thus,
\[
W(m,n;a,b,p,q)=\sum_{k = 1}^n {k^m w_k } = \left. {\frac{{\partial ^m }}{{\partial y^m }}\frac{{(e^{y(n + 2)} qw_n  - e^{y(n + 1)} w_{n + 1}  + a - e^y (ap - b))}}{{e^{2y} q - e^y p + 1}}} \right|_{y = 0} - a\,\delta _{m,0};
\]
in which we can identify
\[
\mathcal P_1 (m,n;p,q) = \left. {\frac{{\partial ^m }}{{\partial y^m }}\frac{{e^{y(n + 2)} q}}{{e^{2y} q - e^y p + 1}}} \right|_{y = 0}, 
\]

\[
\mathcal P_2 (m,n;p,q) = \left. {\frac{{\partial ^m }}{{\partial y^m }}\frac{{ - e^{y(n + 1)} }}{{e^{2y} q - e^y p + 1}}} \right|_{y = 0}, 
\]

\[
\mathcal C(m;a,b,p,q) = \left. {\frac{{\partial ^m }}{{\partial y^m }}\frac{{a - e^y (ap - b)}}{{e^{2y} q - e^y p + 1}}} \right|_{y = 0} - a\,\delta _{m,0}. 
\]
The identity stated in the theorem now follows when we perform the indicated differentiations. Observe that $\mathcal C(m;a,b,p,q)$ can be obtained directly from
\[
\mathcal C(m;a,b,p,q) =  - a\mathcal P_1 (m,0;p,q) - b\mathcal P_2 (m,0;p,q).
\]

\end{proof}
Note that
\[
S(m,n) = W(m,n;0,1,1, - 1),
\]

\[
T(m,n) = W(m,n;2,1,1, - 1),
\]

\[
P_1 (m,n) = \mathcal P_1 (m,n;1, - 1),
\]

\[
P_2 (m,n) = \mathcal P_2 (m,n;1, - 1),
\]

\[
C(m) = \mathcal C(m;0,1,1,-1),
\]

\[
K(m) = \mathcal C(m;2,1,1,-1).
\]
Theorem~\ref{thm.fnhvpm9} can be generalized to $W(m,n,r;a,b,p,q) = \sum_{k = 1}^n {k^m w_{k + r} } $ in a straightforward manner. Using \eqref{eq.bcycobv} and \eqref{eq.nfmkcaz}, we have
\begin{equation}\label{eq.qvo5ipe}
\mathbb A\tau ^r Q(\tau ,m,n) + \mathbb B\sigma ^r Q(\sigma ,m,n) = \delta _{m,0} w_r  + \sum_{k = 1}^n {k^m w_{k + r} }.
\end{equation}
But, from \eqref{eq.idod6y1} we obtain
\begin{equation}\label{eq.szdpd2l}
\begin{split}
&\mathbb A\tau ^r Q(\tau ,m,n) + \mathbb B\sigma ^r Q(\sigma ,m,n) \\
&\qquad=\left. \frac{{\partial ^m }}{{\partial y^m }}\left( {\frac{{e^{y(n + 2)} qw_{n + r}  - e^{y(n + 1)} w_{n + r + 1}  - e^y qw_{r - 1}  + w_r }}{{e^{2y} q - e^y p + 1}}} \right)\right|_{y=0}.
\end{split}
\end{equation}
From \eqref{eq.qvo5ipe} and \eqref{eq.szdpd2l}, it folows that
\[
\begin{split}
W(m,n,r;a,b,p,q) &= \sum_{k = 1}^n {k^m w_{k + r} } \\
&= \left.\frac{{\partial ^m }}{{\partial y^m }}\left( {\frac{{e^{y(n + 2)} qw_{n + r}  - e^{y(n + 1)} w_{n + r + 1}  - e^y qw_{r - 1}  + w_r }}{{e^{2y} q - e^y p + 1}}} \right)\right|_{y=0} - \delta _{m,0} w_r,
\end{split}
\]
which can be written as
\begin{equation}\label{eq.ush6th4}
\begin{split}
W(m,n,r;a,b,p,q)&=\sum_{k = 1}^n {k^m w_{k + r} }\\
&  = \mathcal P_1 (m,n;p,q)w_{n + r}  + \mathcal P_2 (m,n;p,q)w_{n + r + 1}  + \mathcal C(m,r;a,b,p,q);
\end{split}
\end{equation}
where $\mathcal P_1$ and $\mathcal P_2$ are as given in Theorem~\ref{thm.fnhvpm9} and $\mathcal C(m,r;a,b,p,q)$ is given by
\[
\mathcal C(m,r;a,b,p,q) = \left.\frac{{\partial ^m }}{{\partial y^m }}\left( {\frac{{ - e^y qw_{r - 1}  + w_r }}{{e^{2y} q - e^y p + 1}}} \right)\right|_{y=0} - \delta _{m,0} w_r,
\]
or more directly by
\[
\mathcal C(m,r;a,b,p,q) =  - w_r P_1 (m,0;p,q) - w_{r + 1} P_2 (m,0;p,q),
\]
yielding
\[
(q - p + 1)\mathcal C (m,r;a,b,p,q) = -2^mqw_r + w_{r+1} - \sum_{j = 0}^{m - 1} {\binom mj(2^{m - j} q - p)\mathcal C (j,r;a,b,p,q)}.
\]
In Theorem~\ref{thm.siyh5ty} we will generalize Theorem~\ref{thm.gmjtnyz} to Horadam sequences with $p=1$.

Let $p=1$ in \eqref{eq.vhrb5b3} so that we have the second order sequence $(w^*_j(a,b;q))=\left(w_j(a,b; 1, q)\right)$ defined by
\[
w^*_0  = a,\,w^*_1  = b;\,w^*_j  = w^*_{j - 1}  - qw^*_{j - 2}\, (j \ge 2);
\]
where $a$, $b$ and $q$ are arbitrary complex numbers, with $q\ne 0$.
We have the Binet formula
\[
w^*_j = \mathbb A \tau^j + \mathbb B \sigma^j,\quad j\in \mathbb Z,
\]
where
\[
\tau=\frac{1+\sqrt{1 - 4q}}2,\quad \sigma=\frac{1 - \sqrt{1 - 4q}}2,
\]
\[
\tau+\sigma=1,\quad\tau-\sigma=\sqrt{1 - 4q}=\Delta\quad\mbox{and }\tau\sigma=q;
\]
and where 
\[
\mathbb A=\frac{{b - a\sigma}}{{\tau  - \sigma }},\quad \mathbb B=\frac{{a\tau  - b}}{{\tau  - \sigma }}.
\]
In particular we have two special Lucas sequences $u_j(q)=(w^*_j(0,1;q))$, \mbox{$v_j(q)=(w^*_j(2,1;q))$}:
\begin{gather}
u_0  = 0,\,u_1  = 1;\,u_j  = u_{j - 1}  - qu_{j - 2}\, (j \ge 2);\\
\nonumber\\
v_0  = 2,\,u_1  = 1;\,v_j  = v_{j - 1}  - qv_{j - 2}\, (j \ge 2);
\end{gather}
so that
\[
u_j=\frac{\tau ^j - \sigma ^j}{\tau - \sigma},\quad v_j=\tau ^j + \sigma ^j.
\]
In Theorem \ref{thm.siyh5ty} we give a closed form for $\Omega (m,n;a,b,q)=\sum_{k=0}^nk^mw^*_k$ but first we state a couple of lemmas.
\begin{lemma}[Adegoke et al. {\cite[Lemma 1]{adegoke21}}]
For integer $j$,
\[
\mathbb A\tau ^j  - \mathbb B\sigma ^j  = \frac{{w_{j + 1}  - qw_{j - 1} }}{\Delta }.
\]
\end{lemma}
\begin{lemma}\label{lem.mp76dtn}
Let $s$ be a non-negative integer. Then,
\begin{gather}
\mathbb AA_s (\tau ) + \mathbb BA_s (\sigma ) = \sum_{t = 0}^s {A(s,t)w^* _t },\\
\nonumber\\
\mathbb AA_s (\tau ) - \mathbb BA_s (\sigma ) = \frac{1}{\Delta }\sum_{t = 0}^s {A(s,t)(w^* _{t + 1}  - qw^* _{t - 1} )}.
\end{gather}
\end{lemma}
\begin{proof}
We have
\[
\begin{split}
\mathbb AA_s (\tau ) + \mathbb BA_s (\sigma ) &= \mathbb A\sum_{t = 0}^s {A(s,t)\tau ^t }  + \mathbb B\sum_{t = 0}^s {A(s,t)\sigma ^t }\\ 
 &= \sum_{t = 0}^s {A(s,t)(\mathbb A\tau ^t  + \mathbb B\sigma ^t )}  = \sum_{t = 0}^s {A(s,t)w^* _t } ;
\end{split}
\]

\[
\begin{split}
\mathbb AA_s (\tau ) - \mathbb BA_s (\sigma ) &= \sum_{t = 0}^s {A(s,t)(\mathbb A\tau ^t  - \mathbb B\sigma ^t )}\\
&  = \sum_{t = 0}^s {A(s,t)\frac{w^* _{t + 1} - qw^* _{t - 1}}\Delta } .
\end{split}
\]
\end{proof}
\begin{lemma}\label{lem.x8yix44}
Let $r$ and $s$ be non-negative integers. Then,
\[
\begin{split}
&\mathbb A\tau ^r A_s (\tau ) + \mathbb B\sigma ^r A_s (\sigma )\\
 &= \frac{{v_r }}{2}\sum_{t = 0}^s {A(s,t)w^* _t }  + \frac{{u_r }}{2}\sum_{t = 0}^s {A(s,t)(w^* _{t + 1}  - qw^* _{t - 1} )}. 
\end{split}
\]
\end{lemma}
\begin{proof}
Proceeding as in the proof of Lemma \ref{lem.bvh0lkh}, we establish
\[
\begin{split}
&\mathbb A\tau ^r A_s (\tau ) + \mathbb B\sigma ^r A_s (\sigma )\\
&\quad= \frac{{v_r }}{2}(\mathbb AA_s (\tau ) + \mathbb BA_s (\sigma )) + \frac{{u_r \Delta }}{2}(\mathbb AA_s (\tau ) - \mathbb BA_s (\sigma )),
\end{split}
\]
and hence the stated identity via Lemma \ref{lem.mp76dtn}.
\end{proof}
\begin{theorem}\label{thm.siyh5ty}
Let $m$ and $n$ be non-negative integers. Let $(w^*_j(a,b;q))$ be the second order sequence whose terms are given by $w^*_0  = a,\,w^*_1  = b;\,w^*_j  = w^*_{j - 1}  - qw^*_{j - 2}\, (j \ge 2)$. Let $(u_j(q))=(w^*_j(0,1;q))$, $(v_j(q))=(w^*_j(2,1;q))$. Then,
\[
\begin{split}
&\Omega (m,n;a,b,q) = \sum_{k = 1}^n {k^m w^* _k }\\
&\qquad= -\delta_{m,0}a - n^m \frac{{w^* _{n + 2} }}{q} + \frac{{v_{m + 1} }}{{2q^{m + 1} }}\sum_{j = 0}^m {A(m,j)w^* _j }  + \frac{{u_{m + 1} }}{{2q^{m + 1} }}\sum_{j = 0}^m {A(m,j)(w^* _{j + 1}  - qw^* _{j - 1} )}\\
&\qquad\qquad- \sum_{s = 1}^m {n^{m - s} \binom ms\frac{{v_{n + s + 1} }}{{2q^{s + 1} }}\sum_{j = 1}^s {A(s,j)w^* _j } }  - \sum_{s = 1}^m {n^{m - s} \binom ms\frac{{u_{n + s + 1} }}{{2q^{s + 1} }}\sum_{j = 1}^s {A(s,j)(w^* _{j + 1}  - qw^* _{j - 1} )} }.
\end{split}
\]
\end{theorem}
\begin{proof}
From \eqref{eq.k7m77ev} we have
\[
\begin{split}
&\sum_{k=0}^n{k^mw^*_k}=\delta_{m,0}w^*_0 + \sum_{k=1}^n{k^mw^*_k} = \mathbb AQ(\tau ,m,n) + \mathbb BQ(\sigma ,m,n)\\
&= -n^m \left( { \mathbb A \frac{{\tau ^{n + 1} }}{\sigma } + \mathbb B\frac{{\sigma ^{n + 1} }}{\tau }} \right) + \left( {\mathbb A\frac{{A_m (\tau )}}{{\sigma ^{m + 1} }} + \mathbb B \frac{{A_m (\sigma )}}{{\tau ^{m + 1} }}} \right)\\
&\qquad- \sum_{s = 1}^m {n^{m - s} \binom ms\left( {\mathbb A\tau ^n \frac{{A_s (\tau )}}{{\sigma ^{s + 1} }} + \mathbb B \sigma ^n \frac{{A_s (\sigma )}}{{\tau ^{s + 1} }}} \right)}\\
& =- n^m \frac{(\mathbb A\tau ^{n + 2}  + \mathbb B\sigma ^{n + 2} )}{\tau\sigma} + \frac{{\mathbb A\tau ^{m + 1} A_m (\tau ) + \mathbb B\sigma ^{m + 1} A_m (\sigma )}}{{(\tau \sigma )^{m + 1} }}\\
&\qquad- \sum_{s = 1}^m {n^{m - s} \binom ms\left( {\frac{{\mathbb A\tau ^{n + s + 1} A_s (\tau ) + \mathbb B\sigma ^{n + s + 1} A_s (\sigma )}}{{(\tau \sigma )^{s + 1} }}} \right)}, 
\end{split}
\]
from which, using the Binet formula and Lemma \ref{lem.x8yix44}, we obtain the stated identity. 
\end{proof}
In Theorem~\ref{thm.viqym3m} we will specialize the result stated in Theorem~\ref{thm.siyh5ty} to the particular Lucas sequences $(u_j(q))$ and $(v_j(q))$. The following Lemma is required for this purpose.
\begin{lemma}
If $j$ and $s$ are non-negative integers, then,
\begin{gather}
\tau ^s A_j (\tau ) - \sigma ^s A_j (\sigma ) = \Delta \sum_{t = 0}^j {A(j,t)u_{t + s} },\label{eq.w0kqjdb}\\
\tau ^s A_j (\tau ) + \sigma ^s A_j (\sigma ) = \sum_{t = 0}^j {A(j,t)v_{t + s} }\label{eq.j8xowcz}. 
\end{gather}
\end{lemma}
\begin{proof}
The proof parallels that of Lemma \ref{lem.bvh0lkh}.
\end{proof}
\begin{theorem}\label{thm.viqym3m}
If $m$ and $n$ are non-negative integers, then,
\begin{equation}\label{eq.rswhtii}
\begin{split}
\Omega(m,n;0,1,q) &= \sum_{k = 1}^n {k^m u_k }\\ 
 &= -\delta_{m,0}u_0 - n^m \frac{u_{n + 2}}q  + \frac1{q^{m + 1}} \sum_{j = 0}^m {A(m,j)u_{j + m + 1} }\\
&\quad- \sum_{s = 1}^m {\frac1{q^{s + 1}} \binom msn^{m - s} \sum_{j = 1}^s {A(s,j)u_{j + n + s + 1} } }, 
\end{split}
\end{equation}

\begin{equation}\label{eq.ms4d9a5}
\begin{split}
\Omega(m,n;2,1,q) &= \sum_{k = 1}^n {k^m v_k }\\ 
 &=-\delta_{m,0}v_0 - n^m \frac{v_{n + 2}}q  + \frac1{q^{m + 1}} \sum_{j = 0}^m {A(m,j)v_{j + m + 1} }\\
&\quad- \sum_{s = 1}^m {\frac1{q^{s + 1}} \binom msn^{m - s} \sum_{j = 1}^s {A(s,j)v_{j + n + s + 1} } }, 
\end{split}
\end{equation}
where $A(i,j)$ are Eulerian numbers defined in \eqref{eq.pr9b061}.
\end{theorem}
\begin{proof}
Observe that Theorem \ref{thm.viqym3m} is a corollary to Theorem \ref{thm.siyh5ty}. However, it is easier to prove the identities directly. We have
\[
\begin{split}
&\Delta\sum_{k=0}^n{k^mu_k}=\Delta\delta_{m,0}u_0 + \Delta\sum_{k=1}^n{k^mu_k}=Q(\tau ,m,n) - Q(\sigma ,m,n)\\
& =- n^m \frac{(\tau ^{n + 2}  - \sigma ^{n + 2} )}{\tau\sigma} + \frac{{\tau ^{m + 1} A_m (\tau ) - \sigma ^{m + 1} A_m (\sigma )}}{{(\tau \sigma )^{m + 1} }}\\
&\qquad- \sum_{s = 1}^m {n^{m - s} \binom ms\left( {\frac{{\tau ^{n + s + 1} A_s (\tau ) - \sigma ^{n + s + 1} A_s (\sigma )}}{{(\tau \sigma )^{s + 1} }}} \right)}, 
\end{split}
\]
from which, using~\eqref{eq.w0kqjdb}, identity~\eqref{eq.rswhtii} follows. The proof of \eqref{eq.ms4d9a5} is similar. We use
\[
\sum_{k=0}^n{k^mv_k}=\delta_{m,0}v_0 + \sum_{k=1}^n{k^mv_k}=Q(\tau ,m,n) + Q(\sigma ,m,n).
\]
\end{proof}
Next, in Theorem~\eqref{thm.adtdjh1}, we provide a generalization of Theorem~\ref{thm.siyh5ty} to the evaluation of $\mathcal W(m,n,h,r;a,b,p,q) = \sum_{k = 1}^n {V_h^{- k}k^m w_{hk + r}}$, where $(V_j(p,q))=(w_j(2,p;p,q))$ is the Lucas sequence of the second kind.

The numbers $V_j$ are given explicitly by
\[
V_j(p,q)=\tau(p,q)^j + \sigma(p,q)^j ,
\]
where $\tau(p,q)$ and $\sigma(p,q)$ are as given in \eqref{eq.ufpxoag}.
\begin{theorem}\label{thm.adtdjh1}
Let $m$, $n$, $h$ be non-negative integers and $r$ any integer. Then,
\[
\begin{split}
\mathcal W(m,n,r,h;a,b,p,q) &= \sum_{k = 1}^n {\frac{{k^m w_{hk + r} }}{{V_h^k }}}\\
&=  - w_r \delta _{m,0}  - \frac{{n^m w_{h(n + 2) + r} }}{{q^h V_h^n }} + \left( {\frac{{V_h }}{{q^h }}} \right)^{m + 1} \sum_{j = 0}^m {A(m,j)\frac{{w_{h(j + m + 1) + r} }}{{V_h^j }}}\\ 
&\qquad - \sum_{s = 1}^m {\binom ms\frac{{n^{m - s} }}{{q^{h(s + 1)} }}\sum_{j = 1}^s { A(s,j)\frac{w_{h(j + n + s + 1) + r} }{V_h^{j + n - s - 1}}} } .
\end{split}
\]
\end{theorem}
\begin{proof}
From \eqref{eq.bcycobv} and \eqref{eq.nfmkcaz} we have
\begin{equation}\label{eq.f8hb60p}
\begin{split}
&\mathbb A\tau ^r Q(\tau ^h /V_h ,m,n) + \mathbb B\sigma ^r Q(\sigma ^h /V_h ,m,n)\\
&\qquad= \sum_{k = 0}^n {k^m V_h^{ - k} (\mathbb A\tau ^{hk + r}  + \mathbb B\sigma ^{hk + r} )}\\ 
&\qquad= \sum_{k = 0}^n {k^m V_h^{ - k} w_{hk + r} }  = \delta _{m,0} w_r  + \sum_{k = 1}^n {k^m V_h^{ - k} w_{hk + r} }.
\end{split}
\end{equation}
From~\eqref{eq.k7m77ev}, using \eqref{eq.qjlois5} and \eqref{eq.nfmkcaz}, we find
\begin{equation}\label{eq.v1sjrfv}
\begin{split}
&\mathbb A\tau ^r Q(\tau ^h /V_h ,m,n) + \mathbb B\sigma ^r Q(\sigma ^h /V_h ,m,n)\\
&= - \frac{{n^m w_{h(n + 2) + r} }}{{q^h V_h^n }} + \sum_{j = 0}^m {\frac{{V_h^{m + 1 - j} }}{{q^{h(m + 1)} }}A(m,j)w_{h(j + m + 1) + r} }\\ 
&\qquad - \sum_{s = 1}^m {\binom ms\frac{{n^{m - s} }}{{q^{h(s + 1)} }}\sum_{j = 1}^s { A(s,j)\frac{w_{h(j + n + s + 1) + r} }{V_h^{j + n - s - 1}}} } .
\end{split}
\end{equation}
Equating \eqref{eq.f8hb60p} and \eqref{eq.v1sjrfv} we obtain the stated expression for $\mathcal W(m,n,r,h;a,b,p,q)$.
\end{proof}
Note that $\Omega (m,n;a,b,q)\equiv \mathcal W(m,n,0,1;a,b,1,q)$.

Comparing equivalent polynomials in $n$ in W(m,n,r;a,b,1,q), identity~\eqref{eq.ush6th4}, and in\\ $\mathcal W(m,n,r,1;a,b,1,q)$ we find the Ledin constant with the restricted Horadam sequence $(w_j(a,b,1,q))$ to be
\begin{equation}\label{eq.hffekpf}
\mathcal C(m,r;a,b,1,q) =  - w_r^* \delta _{m,0}  + \frac{1}{{q^{m + 1} }}\sum\limits_{j = 0}^m {A(m,j)w_{j + m + 1 + r}^* },
\end{equation}
of which \eqref{eq.t5z8cgr} and \eqref{eq.lzeq4vs} are particular cases.

We can derive a Ledin form for the $w_j^*$ sequence by using $r=- n - 1$, $r=-n$, in turn in identity~\eqref{eq.ush6th4}, written for the special Lucas sequence $(u_j(q))=(U_j(1,q))$ to obtain
\[
 - \frac{1}{q}\mathcal P_1 (m,n;1,q) = \sum\limits_{k = 1}^n {k^m u_{k - n - 1} }  - \mathcal C(m, - n - 1;0,1,1,q)
\]

\[
\mathcal P_2 (m,n;1,q) = \sum\limits_{k = 1}^n {k^m u_{k - n} }  - \mathcal C(m, - n;0,1,1,q)
\]
from which with $p=1$, $h=1$ in the identity of Theorem~\ref{thm.adtdjh1} (that is $\mathcal W(m,n,r,1;a,b,1,q)$) we get
\begin{equation}\label{eq.gs19lfx}
\mathcal P_1 (m,n;1,q) = n^m  + q\sum\limits_{s = 1}^m {\binom ms\frac{{n^{m - s} }}{{q^{s + 1} }}\sum\limits_{j = 1}^s {A(s,j)u_{j + s} } },
\end{equation}

\begin{equation}\label{eq.wzziec2}
\mathcal P_2 (m,n;1,q) =  - \frac{{n^m }}{q} - \sum\limits_{s = 1}^m {\binom ms\frac{{n^{m - s} }}{{q^{s + 1} }}\sum\limits_{j = 1}^s {A(s,j)u_{j + s + 1} } }.
\end{equation}
Thus, using \eqref{eq.hffekpf}, \eqref{eq.gs19lfx} and \eqref{eq.wzziec2} in \eqref{eq.ush6th4} gives
\begin{equation}\label{eq.ldtcggu}
\begin{split}
\sum_{k=1}^n{k^mw_{k + r}^*}&=-w_r^*\,\delta_{m,0}+\left(n^m  + q\sum\limits_{s = 1}^m {\binom ms\frac{{n^{m - s} }}{{q^{s + 1} }}\sum\limits_{j = 1}^s {A(s,j)u_{j + s} } }\right)w_{n + r}^*\\
&\qquad -\left(\frac{{n^m }}{q} + \sum\limits_{s = 1}^m {\binom ms\frac{{n^{m - s} }}{{q^{s + 1} }}\sum\limits_{j = 1}^s {A(s,j)u_{j + s + 1} } }\right)w_{n + r + 1}^*\\
&\quad\qquad + \frac{1}{{q^{m + 1} }}\sum\limits_{j = 0}^m {A(m,j)w_{j + m + 1 + r}^* }.
\end{split}
\end{equation}
Note that $S(m,n,r)$ and $T(m,n,r)$ given in \eqref{eq.x7q4h09} and \eqref{eq.cvt8dsd} are special cases of \eqref{eq.ldtcggu}.
 
Our final result is a generalization of Theorem~\ref{thm.siyh5ty} to Horadam sequences with indices in arithmetic progression.
\begin{theorem}
Let $m$, $n$, $h$ be non-negative integers and $r$ any integer. Then,
\[
\begin{split}
&W(m,n,r,h;a,b,p,q)=\sum_{k = 1}^n {k^m w_{hk + r} }\\
&  =  - \delta _{m,0} w_r  - n^m \left( {\frac{{w_{h(n + 1) + r}  - q^h w_{hn + r} }}{{1 - V_h  + q^h }}} \right)\\
&\qquad + \frac1{(1 - V_h  + q^h )^{m + 1} }\sum_{c = 0}^{m + 1} {( - 1)^c \binom{m + 1}cq^{hc} \sum_{j = 0}^m {A(m,j)w_{h(j - c) + r} } } \\
&\qquad\quad- \sum_{s = 1}^m {\binom ms\frac{{n^{m - s} }}{{(1 - V_h  + q^h )^{s + 1} }}\sum_{c = 0}^{s + 1} {( - 1)^c \binom{s + 1}cq^{hc} \sum_{j = 0}^s {A(s,j)w_{h(j - c + n) + r} } } } .
\end{split}
\]
\end{theorem}
\begin{proof}
From \eqref{eq.bcycobv} and \eqref{eq.nfmkcaz} we have
\begin{equation}\label{eq.blxcboh}
\begin{split}
&\mathbb A\tau ^r Q(\tau ^h ,m,n) + \mathbb B\sigma ^r Q(\sigma ^h ,m,n)\\
&\qquad= \sum_{k = 0}^n {k^m (\mathbb A\tau ^{hk + r}  + \mathbb B\sigma ^{hk + r} )}\\ 
&\qquad= \sum_{k = 0}^n {k^m w_{hk + r} }  = \delta _{m,0} w_r  + \sum_{k = 1}^n {k^m w_{hk + r} }.
\end{split}
\end{equation}
Using \eqref{eq.k7m77ev} directly, we find
\begin{equation}\label{eq.iy1fean}
\begin{split}
&\mathbb A\tau ^r Q(\tau ^h ,m,n) + \mathbb B\sigma ^r Q(\sigma ^h ,m,n)\\
&=  - n^m \left( {\mathbb A\frac{{\tau ^{h(n + 1) + r} }}{{1 - \tau ^h }} + \mathbb B\frac{{\sigma ^{h(n + 1) + r} }}{{1 - \sigma ^h }}} \right) + \mathbb A\frac{{A_m (\tau ^h )\tau ^r }}{{(1 - \tau ^h )^{m + 1} }} + \mathbb B\frac{{A_m (\sigma ^h )\sigma ^r }}{{(1 - \sigma ^h )^{m + 1} }}\\
&\quad - \sum_{s = 1}^m {\binom msn^{m - s} \left( {\mathbb A\frac{{A_s (\tau ^h )\tau ^{hn + r} }}{{(1 - \tau ^h )^{s + 1} }} + \mathbb B\frac{{A_s (\sigma ^h )\sigma ^{hn + r} }}{{(1 - \sigma ^h )^{s + 1} }}} \right)}. 
\end{split}
\end{equation}
Using \eqref{eq.qjlois5}, \eqref{eq.nfmkcaz} and the binomial theorem to express the right hand side of \eqref{eq.iy1fean} in terms of $w_j$, $V_j$ and the Eulerian numbers and equating the resulting expression with the right hand side of \eqref{eq.blxcboh}, we obtain the stated result.
\end{proof}
Comparing coefficients of equivalent polynomials in $n$ in W(m,n,r;a,b,p,q), identity~\eqref{eq.ush6th4}, and in $W(m,n,r,1;a,b,p,q)$ gives the Ledin constant with the full Horadam sequence $(w_j(a,b,p,q))$ as
\begin{equation}\label{eq.idsbz8v}
\mathcal C(m,r;a,b,p,q)=- w_r\delta _{m,0} + \frac1{(1 - p  + q )^{m + 1} }\sum_{c = 0}^{m + 1} {( - 1)^c \binom{m + 1}cq^c \sum_{j = 0}^m {A(m,j)w_{j - c + r} } }.
\end{equation}
We now proceed to establish the Ledin form for the Horadam sequence by determining the polynomials $P_1(m,n;p,q)$, $P_2(m,n;p,q)$.

We write \eqref{eq.ush6th4} for the Lucas sequence of the first kind, namely,
\begin{equation}\label{eq.co5mc67}
\begin{split}
W(m,n,r;0,1,p,q)&=\sum_{k = 1}^n {k^m U_{k + r} }\\
&  = \mathcal P_1 (m,n;p,q)U_{n + r}  + \mathcal P_2 (m,n;p,q)U_{n + r + 1}  + \mathcal C(m,r;0,1,p,q);
\end{split}
\end{equation}
Setting $r=-n - 1$, $r=-n$, in turn, in~\eqref{eq.co5mc67} we have
\[
 - \frac{1}{q}\mathcal P_1 (m,n;p,q) = \sum\limits_{k = 1}^n {k^m U_{k - n - 1} }  - \mathcal C(m, - n - 1;0,1,p,q),
\]

\[
\mathcal P_2 (m,n;p,q) = \sum\limits_{k = 1}^n {k^m U_{k - n} }  - \mathcal C(m, - n;0,1,p,q),
\]
from which, using $W(m,n,r,1;a,b,p,q)$, we get
\begin{equation}\label{eq.zesuvtj}
\mathcal P_1 (m,n;p,q)=\frac{n^mq}{1 - p + q} + q\sum_{s = 1}^m {\binom ms\frac{{n^{m - s} }}{{(1 - p  + q )^{s + 1} }}\sum_{c = 0}^{s + 1} {( - 1)^c \binom{s + 1}cq^{c} \sum_{j = 0}^s {A(s,j)U_{j - c - 1} } } },
\end{equation}

\begin{equation}\label{eq.mklyh6b}
\mathcal P_2 (m,n;p,q)=-\frac{n^m}{1 - p + q} - \sum_{s = 1}^m {\binom ms\frac{{n^{m - s} }}{{(1 - p  + q )^{s + 1} }}\sum_{c = 0}^{s + 1} {( - 1)^c \binom{s + 1}cq^{c} \sum_{j = 0}^s {A(s,j)U_{j - c} } } }.
\end{equation}
Thus,
\[
\begin{split}
W(m,n,r;a,b,p,q)&=\sum_{k = 1}^n {k^m w_{k + r} }\\
&  = \mathcal P_1 (m,n;p,q)w_{n + r}  + \mathcal P_2 (m,n;p,q)w_{n + r + 1}  + \mathcal C(m,r;a,b,p,q),
\end{split}
\]
where $\mathcal P_1(m,n;p,q)$ and $\mathcal P_2(m,n;p,q)$ are as stated in \eqref{eq.zesuvtj} and \eqref{eq.mklyh6b} and $\mathcal C(m,r;a,b,p,q)$ is given in \eqref{eq.idsbz8v}.
\section{Conclusion}
In this paper we addressed the Ledin and Brousseau summation problems. Recursive schemes and polynomial forms were established for $S(m,n,r)=\sum_{k = 1}^n {k^m F_{k + r} }$ and $T(m,n,r)=\sum_{k = 1}^n {k^m L_{k + r} }$ for non-negative integers $m$ and $n$ and arbitrary integer $r$. The study was extended to the general linear second order sequence $w_j(a,b;p,q)$. Recursive procedures and polynomial forms were established for the special restricted case $w_j(a,b;1,q)$ as well as for the general case, including sequences with indices in arithmetic progression. Ledin's suggestions in the concluding part of his paper remain fertile grounds for future research. It would be interesting to extend the study to $S(m,n,s,r)=\sum_{k=1}^n{k^mF_{k + r}^s}$. Ollerton and Shannon~\cite{ollerton20} also provided some areas of further research, such as exploring the sum $\sum_{k=1}^n{f(m,k)F_k}$ or proving the conjectures in their paper~\cite{shannon21}. Yet another area with much prospect would be generalizations to non-homogeneous Fibonacci/Lucas/Horadam sequences or higher order sequences.
\section{Acknowledgement}
The author is much indebted to the referee whose deep insight, useful suggestions and detailed constructive comments have helped to significantly improve the presentation.

\hrule

\bigskip

\bigskip
\noindent Concerned with sequences: A000032, A000045, A000129, A001045, A001582, A002450, A014551

\bigskip
\hrule
\bigskip

\vspace*{+.1in}
\noindent



\end{document}